\documentclass[12pt, reqno]{amsart}
\usepackage{amsmath}
\usepackage{amssymb}
\usepackage{amsthm}
\usepackage{mathrsfs}
\usepackage[abbrev]{amsrefs}
\usepackage[utf8]{inputenc}
\usepackage{enumitem}
\usepackage{bm}

\newtheorem{thm}{}[section]
\newtheorem{theorem}[thm]{Theorem}
\newtheorem{corollary}[thm]{Corollary}
\newtheorem{lemma}[thm]{Lemma}
\newtheorem{proposition}[thm]{Proposition}

\theoremstyle{definition}
\newtheorem{definition}[thm]{Definition}
\theoremstyle{remark}
\newtheorem{remark}[thm]{Remark}

\newtheorem{question}[thm]{Question}

\numberwithin{equation}{section}
\allowdisplaybreaks

\newcommand{\Nt}{\ensuremath{\mathcal{N}}}

\newcommand{\Ts}{\ensuremath{\mathcal{T}}}
\newcommand{\TT}{\ensuremath{\mathbb{T}}}

\newcommand{\tb}{\ensuremath{\bm{t}}}

\newcommand{\vv}{\ensuremath{\bm{v}}}
\newcommand{\xx}{\ensuremath{\bm{x}}}
\newcommand{\yy}{\ensuremath{\bm{y}}}

\newcommand{\uu}{\ensuremath{\bm{u}}}
\newcommand{\XX}{\ensuremath{\mathbb{X}}}
\newcommand{\YY}{\ensuremath{\mathbb{Y}}}
\newcommand{\BB}{\ensuremath{\mathcal{B}}}

\newcommand{\NN}{\ensuremath{\mathbb{N}}}

\newcommand{\UU}{\ensuremath{\mathbb{U}}}

\newcommand{\supp}{\operatorname{supp}}

\AtBeginDocument{
\def\MR#1{}
}

\begin{document}

\title[Splitting unconditional bases]{Uniqueness of unconditional basis of $\ell_{2}\oplus \Ts^{(2)}$}

\author[F. Albiac]{Fernando Albiac}
\address{Department of Mathematics, Statistics and Computer Sciences, and InaMat$^2$\\ Universidad P\'ublica de Navarra\\
Pamplona 31006\\ Spain}
\email{fernando.albiac@unavarra.es}

\author[J. L. Ansorena]{Jos\'e L. Ansorena}
\address{Department of Mathematics and Computer Sciences\\
Universidad de La Rioja\\
Logro\~no 26004\\ Spain}
\email{joseluis.ansorena@unirioja.es}



\begin{abstract}
We provide a new extension of Pitt's theorem for compact operators between quasi-Banach lattices, which permits to describe unconditional bases of finite direct sums of Banach spaces $\XX_{1}\oplus\dots\oplus\XX_{n}$ as direct sums of unconditional bases of its summands. The general splitting principle we obtain yields, in particular, that if each $\XX_{i}$ has a unique unconditional basis (up to equivalence and permutation), then $\XX_{1}\oplus \cdots\oplus\XX_{n}$ has a unique unconditional basis too. Among the novel applications of our techniques to the structure of Banach and quasi-Banach spaces we have that the space $\ell_2\oplus \Ts^{(2)}$ has a unique unconditional basis.\end{abstract}

\thanks{F. Albiac acknowledges the support of the Spanish Ministry for Science and Innovation under Grant PID2019-107701GB-I00 for \emph{Operators, lattices, and structure of Banach spaces}. F. Albiac and J.~L. Ansorena acknowledge the support of the Spanish Ministry for Science, Innovation, and Universities under Grant PGC2018-095366-B-I00 for \emph{An\'alisis Vectorial, Multilineal y Aproximaci\'on}.}

\maketitle

\section{Introduction and background}
\noindent
Recall that a quasi-Banach space (in particular a Banach space) $\XX$ with a semi-normalized unconditional basis $\BB=(\xx_{n})_{n\in \Nt}$  is said to have a unique up to permutation (UTAP for short) unconditional basis if any other semi-normalized unconditional basis of $\XX$ is equivalent to a permutation of $\BB$. A  very natural problem in  the theory  is the following:

\begin{question}\label{question:gluel2}
Does the space $\XX\oplus\YY$ have a UTAP unconditional basis provided the Banach spaces $\XX$ and $\YY$ do?
\end{question}

In 1976, Edelstein and Wojtaszczyk (\cite{EdelWoj1976}) gave an affirmative answer to Question~\ref{question:gluel2} in the case when $\XX$ and $\YY$ are one of the only three Banach spaces with a unique unconditional basis (up to equivalence), namely $\ell_{1}$, $\ell_{2}$, and $c_{0}$. These positive results motivated Bourgain et al.\ to study infinite direct sums of these spaces in their 1985 \emph{Memoir} \cite{BCLT1985} with the aim to classify all Banach spaces with a (UTAP) unconditional basis. As an indirect outcome of their work, in relation to Question~\ref{question:gluel2} they obtained that the Banach spaces $\ell_{1}\oplus c_{0}(\ell_{1}^{(n)})$, $c_{0}\oplus \ell_{1}(\ell_{\infty}^{ (n)})$, $\ell_{2}\oplus \ell_{1}(\ell_{2}^{(n)})$, and $\ell_{2}\oplus c_{0}(\ell_{2}^{(n)})$ have a (UTAP) unconditional basis.  However, all hopes of attaining a satisfactory classification were shattered when they found an unexpected space of a
totally different character, namely the 2-convexification of Tsirelson's space, denoted $\Ts^{(2)}$, with this property. The fact that $\Ts^{(2)}$ was a useful example in several other contexts had already been previously noticed in \cite{FigielLindenstraussMilman1977} and \cite{Johnson1979-80}.

This article is a direct continuation of \cite{AAW2020b}, where the authors show that an amalgamation of techniques that have their origin in the papers \cites{DLMR2000,Woj1978} paves the way to obtaining a positive answer to Question~\ref{question:gluel2} in some new cases (see \cite{AAW2020b}*{Theorem 4.4}). Roughly speaking this result works when there is a ``gap'' between
the lattice concavity of $\XX$ and the lattice convexity of $\YY$,
i.e., in the case when there exist $1\le q<r\le\infty$ such that the lattice structure induced by their (unique) unconditional bases on $\XX$ and $\YY$
satisfies a lower $q$-estimate and an upper $r$-estimate, respectively.
This yields, for instance, that if $\XX$ is either $\ell_1$ or $\Ts$, $\YY$ is either $\ell_2$ or $\Ts^{(2)}$, and $\UU$ is either $c_0$ or $\Ts^*$, then the spaces $\XX\oplus \YY$, $\XX\oplus \UU$, $\YY\oplus \UU$, and $\XX\oplus \YY\oplus \UU$ all have a UTAP unconditional basis (see \cite{AlbiacAnsorena2020b}*{Corollary 6.2} and \cite{AAW2020b}*{Theorem 4.7}).

Since the optimal lattice concavity of $\ell_2$ and the optimal lattice convexity of $\Ts^{(2)}$ agree (and it is 2), the space $\ell_2\oplus\Ts^{(2)}$ is out of the scope of the methods from \cite{AAW2020b}. And so are the spaces $\ell_1\oplus\Ts$ and  $c_0\oplus\Ts^*$. However, since $\ell_1$, $c_0$, $\Ts$, and $\Ts^*$ are lattice anti-Euclidean spaces, it can be shown that any direct sum built with some of these spaces has a UTAP unconditional basis (see \cite{AlbiacAnsorena2020b}*{Corollary 6.2}). Again, as $\ell_2$ and $\Ts^{(2)}$ are ``highly'' Euclidean, they are orthogonal to the methods from \cite{AlbiacAnsorena2020b}.

This paper fils this gap in the theory by developing novel techniques which permit to prove, in particular, that $\ell_2\oplus\Ts^{(2)}$ has a UTAP unconditional basis.

All our notations and unexplained terminology are those of \cite{AAW2020b}, where the reader will also find a thorough exposition of the background problems that motivate this work.

\section{Yet another generalization of Pitt's theorem for operators between quasi-Banach lattices}\label{sect:SpBanach}
\noindent
The classical Pitt's theorem \cite{Pitt1936} asserts that  every continuous linear operator from $\ell_{p}$ into $\ell_{q}$,  $1\le q<p\le \infty$, is compact. Since the publication of this result, many pairs $(\XX, \YY)$ of Banach spaces with the property that every operator between $\XX$ and $\YY$ is compact have been isolated and many variants of this result in different settings have been studied by several authors. Defant et al.\ addressed in \cite{DLMR2000} the same question for operators between quasi-Banach sequence spaces, i.e, quasi-Banach spaces  endowed with a lattice structure induced by an unconditional basis. They proved that every continuous linear operator from $\XX$ into $\YY$ is compact provided  there is a gap between the lattice estimates of $\XX$ and $\YY$.

This section is geared towards obtaining an extension of \cite{DLMR2000}*{Theorem 1} in the case when the optimal $p$ so that $\XX$ satisfies an upper $p$-estimate and the optimal $q$ so that $\YY$ satisfies a lower $q$-estimate coincide. To that end we start with a lemma that follows ideas from \cite{DLMR2000}.

\begin{lemma}\label{lem:compactBBS}
Let $T\colon\YY\to\XX$ be a bounded non-compact linear operator between quasi-Banach spaces. Suppose that $\XX$ has a Schauder basis $\BB_x$. Then there exist a semi-normalized block basic sequence $\BB_u$ with respect to $\BB_x$ and a semi-normalized basic sequence $\BB_v$ in $\YY$ such that  $\BB_u$ and $T(\BB_v)$ are congruent in $\XX$. Moreover, if $\YY$ has a Schauder basis $\BB_y$, we can choose $\BB_v$ to be a block basic sequence with respect to $\BB_y$.
\end{lemma}

\begin{proof}As the proof of the lemma without assuming that $\YY$ has a Schauder basis is similar and easier, we will carry out the proof of the ``moreover'' part.

By assumption, there is $(f_n)_{n=1}^\infty$ in $B_\YY$ such that $(T(f_n))_{n=1}^\infty$ has no convergent subsequences. For each $k\in\NN$, the scalar sequences $(\yy_k^*(f_n))_{n=1}^\infty$ and $(\xx_k^*(T(f_n)))_{n=1}^\infty$ are bounded. Using a Cantor's diagonal argument, passing to a subsequence we assume that these sequences converge for every $k\in\NN$. Since $(T(f_n))_{n=1}^\infty$ is not a Cauchy sequence, there exist $\delta>0$ and an increasing sequence $(n_j)_{j=1}^\infty$ in $\NN$ such that
\[
\Vert T(f_{n_{2j-1}}) - T(f_{n_{2j}})\Vert\ge\delta, \quad j\in\NN.
\]
Set $g_j=f_{n_{2j-1}}- f_{n_{2j}}$ for $j\in\NN$. The sequences $(g_j)_{j=1}^\infty$ and $(T(g_j))_{j=1}^\infty$ are semi-normalized with
\[
\lim_j \yy_k^*(f_j)=0=\lim_j \xx_k^*(T(f_j))
\]
for all $k\in\NN$. Let $(\varepsilon_j)_{j=1}^\infty$ be a null sequence of positive numbers. By the gliding hump technique, passing to a subsequence we infer that there are block basic sequences $(\uu_j)_{j=1}^\infty$ and $(\vv_j)_{j=1}^\infty$ with respect to $\BB_x$ and $\BB_y$ respectively such that
\[
\Vert g_j -\vv_j\Vert\le\varepsilon_j, \quad \Vert T(g_j) -\uu_j\Vert\le\varepsilon_j, \quad j\in\NN.
\]
Assume without loss of generality that $\XX$ is $p$-Banach for some $0<p\le 1$. Then
\begin{align*}
\Vert T(\vv_j) - \uu_j\Vert^p
&\le \Vert T(\vv_j) - T(g_j) +T(g_j) - \uu_j\Vert^p\\
& \le \Vert T(\vv_j) - T(g_j)\Vert^p +\Vert T(g_j) - \uu_j\Vert^p\\
&\le (1+\Vert T\Vert^p) \varepsilon^p_j.
\end{align*}
Choosing $(\varepsilon_j)_{j=1}^\infty$ small enough, the small perturbation principle yields that $\BB_u=(\uu_j)_{j=1}^\infty$ and $\BB_v=(\vv_j)_{j=1}^\infty$ satisfy the desired properties.
\end{proof}

\begin{remark}\label{rmk:lpembeds}
Applying Lemma~\ref{lem:compactBBS} with $\YY=\ell_q$, $0<q\le\infty$, yields that if $\XX$ contains an isomorphic copy of $\ell_q$ (we replace $\ell_\infty$ with $c_0$), then $\XX$ has a block basic sequence equivalent to the canonical $\ell_{q}$-basis.
\end{remark}

\begin{remark}\label{rmk:subspace}
Applying Lemma~\ref{lem:compactBBS} with $\YY\subseteq \XX$ yields that any infinite-dimensional subspace $\YY$ of $\XX$ contains a sequence congruent to a block basic sequence.
\end{remark}

\begin{proposition}\label{prop:1}
Let $T\colon\YY\to\XX$ be a bounded non-compact linear operator between quasi-Banach spaces. Suppose that $\XX$ has an unconditional basis $\BB_x$ that satisfies a lower $q$-estimate for some $0<q<\infty$. Suppose also that:
\begin{enumerate}[label={{(\roman*)}}, leftmargin=*, widest=ii]

\item Either $0<q\le 1$ and $\YY$ is a $q$-Banach space; or

\item\label{prop:C2} $\YY$ has an unconditional basis $\BB_y$ that satisfies an upper $q$-estimate.
\end{enumerate}
Then there is a basic sequence $\BB_v$ in $\YY$ and a basic sequence in $\XX$ finitely disjointly supported with respect to $\BB_x$ which is equivalent to the canonical $\ell_{q}$-basis, to $\BB_v $, and to $T(\BB_v)$. Moreover, if \ref{prop:C2} holds, $\BB_v$ can be chosen to be finitely disjointly supported with respect to $\BB_y$.
\end{proposition}

\begin{proof}
Use Lemma~\ref{lem:compactBBS} to pick a semi-normalized sequence $\BB_u$ disjointly finitely supported with respect to $\BB_x$ and a semi-normalized sequence $\BB_v=(\vv_j)_{j=1}^\infty$ in $\YY$ such that $\BB_u$ is equivalent to $T(\BB_v)$. Then $T(\BB_v)$ dominates the unit vector basis of $\ell_{q}$ and, in turn, the unit vector basis of $\ell_{q}$ dominates $\BB_v$. Hence, there are constants $C_1$ and $C_2$ such that
\begin{align*}
\left( \sum_{j=1}^\infty |\lambda_j|^q\right)^{1/q}
&\le C_1 \left\Vert \sum_{j=1}^\infty \lambda_j T(\vv_j) \right\Vert\\
&\le C_1\Vert T\Vert \left\Vert \sum_{j=1}^\infty \lambda_j \vv_j \right\Vert\\
&\le C_1 C_2 \Vert T\Vert \left( \sum_{j=1}^\infty |\lambda_j|^q\right)^{1/q},
\end{align*}
for any $(\lambda_j)_{j=1}^\infty\in c_{00}$.
\end{proof}

The proof of the following consequence of Proposition~\ref{prop:1} is straightforward.

\begin{corollary}\label{cor:1}
Let $\XX$ be a quasi-Banach space with an unconditional basis $\BB_x$ which satisfies a lower $q$-estimate for some $0<q<\infty$. Let $\YY$ be another quasi-Banach space. Suppose that either $0<q\le 1$ and $\YY$ is a $q$-Banach space, or $\YY$ has an unconditional basis $\BB_y$ which satisfies an upper $q$-estimate. Then every strictly singular operator from $\YY$ into $\XX$ is compact.
\end{corollary}

Suppose  there is a gap between the lattice concavity induced by the unconditional basis $\BB_x$ of a quasi-Banach space $\XX$ and  the lattice  convexity induced by the unconditional basis $\BB_y$ of a quasi-Banach space $\YY$, i.e.,  $\BB_x$ satisfies a  lower $q$-estimate and $\BB_y$ satisfies an upper $r$-estimate for some  $0<q<r\le \infty$. Then,  $\BB_x$ satisfies a lower $q$-estimate, $\BB_y$ satisfies an upper $q$-estimate, and $\ell_q$ does not embed into $\YY$.
Proposition~\ref{prop:1} tells us that this assumption on the bases $\BB_x$ and $\BB_y$ is enough to guarantee the  compactness of operators between $\XX$ and $\YY$ in the spirit of Pitt's theorem. This motivates the following definition.

\begin{definition}\label{def:IC}
We say that a finite family $(\BB_j)_{j=1}^n$ of unconditional bases of respective quasi-Banach spaces $(\XX_j)_{j=1}^n$ is \emph{lattice-estimate disjoint} if the following conditions hold for all $1\le j \le n-1$:
\begin{enumerate}[label={{(\roman*)}}, leftmargin=*, widest=iii]

\item There exists a non-decreasing sequence $(q_j)_{j=1}^{n-1}$ in $(0,\infty]$ such that $\BB_j$ satisfies a lower $q_j$-estimate;

\item $\BB_{j+1}$ satisfies an upper $q_j$-estimate; and

\item either $\XX_j$ or $\XX_{j+1}$  contain no isomorphic copy of $\ell_{q_j}$.

\end{enumerate}
\end{definition}

\begin{corollary}\label{cor:CCC}
Let $\XX$ and $\YY$ be quasi-Banach spaces with respective unconditional bases $\BB_x$ and $\BB_y$. Suppose that the pair $(\BB_x,\BB_y)$ is lattice-estimate disjoint. Then every bounded linear operator from $\YY$ into $\XX$ is compact.
\end{corollary}

For instance, given $0<q<\infty$, the $q$-convexified Tsirelson space $\Ts^{(q)}$ satisfies an upper $q$-estimate and contains no copy of $\ell_q$. Thus, Corollary~\ref{cor:CCC} applies to $\XX=\ell_{q}$ and $\YY=\Ts^{(q)}$ despite the fact that there is no gap between the lattice estimates of both spaces.

\begin{remark}\label{rmk:stopping} Let $(\BB_j)_{j=1}^n$ be a family of lattice-estimate disjoint unconditional bases, and let $(q_j)_{j=1}^{n-1}$ be as in Definition~\ref{def:IC}. Suppose that for some $q\in (0,\infty]$ and $i\in\{2,\dots,n-1\}$ we have $q_{i-1}=q_{i}=q$. Then $\BB_i$ satisfies both a lower $q$-estimate and an upper $q$-estimate so that it is equivalent to the canonical basis of $\ell_{q}$. Therefore, $\XX_{i-1}$ contains no copy of $\ell_q$. This implies that there can be at most two indices $j$ with $q_j=q$.
\end{remark}

\section{Splitting unconditional bases of direct sums of Banach spaces}
\noindent
We say that a finite family $(\XX_j)_{j=1}^n$ of Banach spaces is \emph{splitting for unconditional bases} if every unconditional basis of $\bigoplus_{j=1}^{n} \XX_j$ splits into basic sequences $(\BB_j)_{j=1}^{n}$ with $[\BB_j]\simeq \XX_j$ for each $j=1,\dots,n$.

In 1976, Edelstein and Wojtaszczyk \cite{EdelWoj1976} proved that any finite family $(\ell_{p_{j}})_{j=1}^{n}$ for $1\le p_{j}\le \infty$ (with the convention that $\ell_{\infty}$ means $c_{0}$) is splitting for unconditional bases. This important structural property has been  recently studied for general Banach spaces in \cite{AAW2020b},
where the authors establish sufficient conditions for a finite family of unconditional bases to be splitting    in terms of gaps between lattice estimates of the bases. Our main theorem of this section relies on the refined notion of lattice-estimate disjointness to  improve \cite{AAW2020b}*{Theorem 4.4}.

\begin{theorem}\label{thm:split}
Let $(\XX_j)_{j=1}^n$ be a finite family of Banach spaces. Suppose that each $\XX_j$ has an unconditional basis $\BB_j$ and that $(\BB_j)_{j=1}^n$ is lattice-estimate disjoint. Then $(\XX_j)_{j=1}^n$ is splitting for unconditional bases. In particular, if $\XX_j$ has a UTAP unconditional basis for all $1\le j \le n$, then $\bigoplus_{j=1}^n \XX_j$ has a UTAP unconditional basis.
\end{theorem}

The proof of Theorem~\ref{thm:split} relies on our extension of Pitt's theorem from Section~\ref{sect:SpBanach} and the following three instrumental lemmas.

\begin{lemma}\label{lem:111}
Let $\XX$ and $\YY$ be quasi-Banach spaces with respective unconditional bases $\BB_x$ and $\BB_y$. Suppose that there is $0<q\le\infty$ such that $\BB_x$ satisfies a lower $q$-estimate and $\BB_y$ satisfies an upper $q$-estimate. Then every basic sequence $\BB$ in $\XX\oplus \YY$ disjointly supported with respect to $\BB_x\oplus\BB_y$ has a subsequence equivalent either to a basic sequence in $\XX$ disjointly supported with respect to $\BB_x$, or to a basic sequence in $\YY$ disjointly supported with respect to $\BB_y$.
\end{lemma}

\begin{proof}
Put $\BB=(\uu_n,\vv_n)_{n=1}^\infty$. Passing to a subsequence we can suppose that either $\inf_n \Vert \uu_n \Vert>0$ or $\Vert \uu_n \Vert\le \alpha_n$ for all $n\in\NN$, where $(\alpha_n)_{n=1}^\infty$ is a given sequence of positive scalars. In the former case, $\BB_u=(\uu_n)_{n=1}^\infty$ dominates the canonical basis of $\ell_{q}$ and so $\BB_u$ also dominates $\BB_v$. Consequently, $\BB$ is equivalent to $\BB_u$. In the latter case, the principle of small perturbations yields that $\BB$ is equivalent to $\BB_v$ for a suitable choice of $(\alpha_n)_{n=1}^\infty$.
\end{proof}

\begin{lemma}\label{lem:lpdirectsum}
Let $\XX$ and $\YY$ be quasi-Banach spaces with respective unconditional bases $\BB_x$ and $\BB_y$. Suppose that $\BB_x$ satisfies a lower $q$-estimate and that $\BB_y$ satisfies an upper $q$-estimate for some $0<q\le \infty$. Then, given $0<p<q$ (resp.\ $q<p\le \infty$), the space $\ell_p$ is isomorphic to a subspace of $\XX\oplus\YY$ if and only if $\ell_p$ is isomorphic to a subspace of $\XX$ (resp.\ of $\YY$), with the convention that we replace $\ell_\infty$ with $c_0$ if $p=\infty$.
\end{lemma}

\begin{proof}
Suppose that $\ell_p$ is isomorphic to a subspace of $\XX\oplus\YY$. Then there is a basic sequence in $\XX\oplus\YY$ disjointly supported with respect to $\BB_x\oplus\BB_y$ and equivalent to the canonical basis of $\ell_{p}$ (see Remark~\ref{rmk:lpembeds}). By Lemma~\ref{lem:111}, there is a sequence $\BB$ disjointly supported with respect to either $\BB_x$ or $\BB_y$ and equivalent to the canonical basis of $\ell_{p}$. In the former case, the canonical basis of $\ell_{p}$ dominates the canonical basis of $\ell_{q}$ so that $p\le q$. In the latter case the canonical basis of $\ell_{p}$ is dominated by the canonical basis of $\ell_{q}$ and, hence $p\ge q$. Summing up, if $p<q$ (resp.\ $p>q$), $\BB$ is a basic sequence in $\XX$ (resp.\ in $\YY$) disjointly supported with respect to $\BB_x$ (resp.\ $\BB_y$) and so $[\BB]$ is a subspace of $\XX$ (resp.\ $\YY$) isomorphic to $\ell_p$.
\end{proof}

\begin{lemma}\label{lem:SUBfromcouples}
Let $(\XX_i)_{i\in I}$ be a finite family of quasi-Banach spaces.

\begin{enumerate}[label=(\roman*), leftmargin=*, widest=ii]
\item\label{lem:SUB:1}Suppose that there is a partition $(I_j)_{j\in J}$ of $I$ such that $(\XX_j)_{i\in I_j}$ is splitting for unconditional bases for all $j\in J$ and $(\bigoplus_{i\in I_j} \XX_i)_{j\in J}$ is splitting for unconditional bases. Then $(\XX_i)_{i\in I}$ is splitting for unconditional bases.

\item\label{lem:SUB:2} Suppose there is a bijection $\pi\colon \{1,\dots,|I|\}\to I$ such
for each $2\le s \le |I|$, the pair $(\bigoplus_{j=1}^{s-1} \XX_{\pi(j)} ,\XX_{\pi(s)})$ is splitting for unconditional bases. Then $(\XX_j)_{j\in I}$ is splitting for unconditional bases.
\end{enumerate}
\end{lemma}

\begin{proof}It is straightforward.\end{proof}

The last ingredient we need before we tackle the proof of Theorem~\ref{thm:split} is the following result by Wojtaszczyk.
\begin{theorem}[see \cite{Woj1978}*{Theorem 2.1}]\label{Woj:Split}
Let $\XX$ and $\YY$ be Banach spaces such that every bounded linear operator from $\XX$ into $\YY$ is compact. Then $(\XX,\YY)$ is splitting for unconditional bases.
\end{theorem}

\begin{proof}[Completion of the Proof of Theorem~\ref{thm:split}]
Assume without loss of generality that $\BB_j$ is semi-normalized for all $1\le j \le n$. Let $(q_j)_{j=1}^{n-1}$ be as in Definition~\ref{def:IC}. Before tackling the general case, we will consider a couple of particular cases.

\textsc{Case A.} If $n=2$, the result follows by combining Corollary~\ref{cor:CCC} with Theorem~\ref{Woj:Split}.

\textsc{Case B.} Suppose that $n\ge 3$ and $q_j<q_{j+1}$ for all $1\le j\le n-2$. Since $\bigoplus_{j=1}^{s}\BB_j$ satisfies a lower $q_{s}$-estimate for all $1\le s\le n-1$, combining Lemma~\ref{lem:lpdirectsum} with  Case A we obtain that $(\bigoplus_{j=1}^{s-1}\XX_j,\XX_{s})$ is splitting for unconditional bases for all $2\le s\le n$. Applying Lemma~\ref{lem:SUBfromcouples}~\ref{lem:SUB:2} yields the desired result.

In the general case, we pick an arbitrary element $\alpha$ not belonging to $[1,\infty]$, and we use the convention that $q<\alpha$ for all $q\in[1,\infty]$. Let $A:=\{q_j\colon 1\le j \le n-1\}\cup\{\alpha\}$. For each $q\in A\setminus\{\alpha\}$ set $J_q=\{j\colon q_j=q\}$, $\YY_q=\bigoplus_{j\in J_q} \XX_j$ and $\BB'_q=\bigoplus_{j\in J_q} \BB_j$. Set also $\YY_\alpha=\XX_n$ and $\BB'_\alpha=\BB_n$. If $q$, $r\in A$ are such that $q<r$, then $\BB_q'$ satisfies a lower $q$-estimate and $\BB'_r$ satisfies and upper $q$-estimate. Moreover, by Lemma~\ref{lem:lpdirectsum}, either $\YY_q$ or $\YY_r$ contains no copy of $\ell_q$. Then, Case B yields that $(\YY_q)_{q\in A}$ is splitting for unconditional bases. Combining Remark~\ref{rmk:stopping} with Case A yields that $(\XX_j)_{j\in J_q}$ is splitting for unconditional bases. We conclude the proof by applying Lemma~\ref{lem:SUBfromcouples}~\ref{lem:SUB:1}.
\end{proof}

\section{Applications}

\noindent We close with some applications of Theorem~\ref{thm:split} to the classical theory of Banach and quasi-Banach spaces.

Given $0<r<\infty$ and $1<s\le\infty$, we will denote by $\Ts^{(r)}$ the $r$-convexified Tsirelson space and by $\Ts^{(s)}_*$ the dual space of $\Ts^{(r)}$, where $r=s/(s-1)$. It is known that  for $0<r<\infty$ the space $\Ts^{(r)}$ contains no subsymmetric basic sequence, i.e., no unconditional basic sequence equivalent to all its subsequences. With the help of
our next lemma we will be able to extend this result to dual spaces.

\begin{lemma}\label{lem:AnsoSS}
Let $\XX$ be a Banach space with an unconditional basis $\BB=(\xx_j)_{j=1}^\infty$. Let $\YY$ be the subspace of $\XX^*$ spanned by the biorthogonal functionals $\BB^*=(\xx_j^*)_{j\in\Nt}$ of $\BB$. Suppose that there are a constant $C$ and an increasing sequence $(j_n)_{n=1}^\infty$ in $\NN$ with the following property:
$
\left\Vert \sum_{n=1}^\infty f_n\right\Vert \le C \left\Vert \sum_{n=1}^\infty g_n\right\Vert
$
whenever $(f_n)_{n=1}^\infty$ and $(g_n)_{n=1}^\infty$ in $\XX$ and $(k_n)_{n=1}^\infty$ in $\NN$ satisfy $\Vert f_n\Vert \le \Vert g_n\Vert$, $\supp(f_n)\cup\supp(g_n) \subseteq[k_n, k_{n+1}-1]$, and $j_n\le k_n$ for all $n\in\NN$.
Then, any subsymmetric basis sequence in $\XX$ (resp.\ $\YY$) is equivalent to $(\xx_{k_n})_{n=1}^\infty$ (resp.\ $(\xx^*_{k_n})_{n=1}^\infty$) for some increasing sequence $(k_n)_{n=1}^\infty$ in $\NN$ with $j_n\le k_n$ for all $n\in\NN$.
\end{lemma}

\begin{proof}
Let $\BB_s$ be a subsymmetric sequence in $\XX$ (resp.\ in $\YY$). By \cite{AADK2019}*{Proposition 2.2}, there is a block basic sequence $(\yy_n)_{n=1}^\infty$ with respect to $\BB$ (resp.\ $\BB^*$) equivalent to $\BB_s$. Passing to a subsequence we can suppose that $k_n:=\min(\supp(\yy_n))\ge j_n$ for all $n\in\NN$. Then, by \cite{AAW2020b}*{Lemma 3.9}, $\BB_s$ is equivalent to $(\xx_{k_n})_{n=1}^\infty$ (resp.\ $(\xx^*_{k_n})_{n=1}^\infty$).
\end{proof}

\begin{proposition}\label{prop:SSTs}Let $0<r<\infty$ and $1<s\le\infty$. Then neither $\Ts^{(r)}$ nor $\Ts^{(s)}_*$ contain a subsymmetric basic sequence.
\end{proposition}

\begin{proof}It is known that the unit vector system $(\tb_j)_{j=1}^\infty$ of $\Ts$ satisfies the assumptions in Theorem~\ref{lem:AnsoSS} for any $(j_n)_{n=1}^\infty$ (see \cite{CasShura1989}*{Corollary II.5}). And so does the unit vector system $(\tb_j^{(p)})_{j=1}^\infty$ of $\Ts^{(p)}$ for $0<p<\infty$ since $p$-convexifications inherit this property. Denote by $(\xx_j)_{j=1}^\infty$ the unit vector system of $\Ts^{(s)}_*$ (resp.\ $\Ts^{(r)}$). Assume by contradiction that $\Ts^{(s)}_*$ (resp.\ $\Ts^{(r)}$) has a subsymmetric sequence. By Lemma~\ref{lem:AnsoSS} there is an increasing sequence $(j_n)_{n=1}^\infty$ such that $(\xx_{j_n})_{j=1}^\infty$ is subsymmetric. By duality, $(\tb_{j_n}^{(p)})_{n=1}^\infty$ is subsymmetric as well, where $p=s/(s-1)$ (resp.\ $p=r$). Hence, $(\tb_{j_n})_{n=1}^\infty$ is subsymmetric. But no subbasis of $(\tb_j)_{j=1}^\infty$ is subsymmetric.
\end{proof}

The next two theorems follow easily by combining Theorem~\ref{thm:split} with Proposition~\ref{prop:SSTs}.
\begin{theorem}
Let $A\subseteq[1,\infty]$, $B\subseteq[1,\infty)$, and $D\subseteq (1,\infty]$ be finite sets. Then $(\ell_p)_{p\in A} \sqcup (\Ts^{(r)})_{r\in B} \sqcup (\Ts_*^{(s)})_{s\in D}$ is splitting for unconditional basis.
\end{theorem}

\begin{theorem}\label{thm:11}
Let $A\subseteq\{1,2,\infty\}$, $B\subseteq\{1,2\}$ and $D\subseteq \{ 2,\infty\}$. Then $(\bigoplus_{p\in A} \ell_p) \oplus (\bigoplus_{r\in B} \Ts^{(r)}) \oplus (\bigoplus_{s\in D} \Ts_*^{(s)})$ has a UTAP unconditional basis (we replace $\ell_\infty$ with $c_0$ is $p\in A$).
\end{theorem}

Finally we show some applications of our results to the uniqueness of structure of nonlocally convex quasi-Banach spaces. We note that Theorem~\ref{last} below remains valid if we replace the Hardy space $H_p(\TT^d)$ with any of the spaces with a unique unconditional basis from \cite{AlbiacAnsorena2020b}*{Corollary 6.2}.

\begin{theorem}\label{last}
Let $0<p<1$ and $d\in\NN$. Let $A\subseteq\{1,2,\infty\}$, $B\subseteq\{1,2\}$ and $D\subseteq \{ 2,\infty\}$. Then the space
\[H_p(\TT^d)\oplus(\oplus_{p\in A} \ell_p) \oplus (\oplus_{r\in B} \Ts^{(r)}) \oplus (\oplus_{s\in D} \Ts_*^{(s)})
\] has a UTAP unconditional basis (we replace $\ell_\infty$ with $c_0$ if $p\in A$).
\end{theorem}

\begin{proof}By \cite{Pel1960}*{Lemma 2}, \cite{AAW2020b}*{Theorem 3.12} and \cite{OS2015}*{Proposition 2.2}, $(\bigoplus_{p\in A} \ell_p) \oplus (\bigoplus_{r\in B} \Ts^{(r)}) \oplus (\bigoplus_{s\in D} \Ts_*^{(s)})$ is subprojective. Now the result follows by combining \cite{AAW2020b}*{Theorem 4.1} with Theorem~\ref{thm:11}.
\end{proof}


\end{document}